\documentclass[11pt,a4paper]{article}
\usepackage[utf8]{inputenc}
\usepackage[T1]{fontenc}
\usepackage[english]{babel}
\usepackage[a4paper,left=2.5cm,right=2.5cm,top=2.5cm,bottom=2.5cm]{geometry}
\usepackage{amsmath,amssymb,amsthm}
\usepackage{mathtools}
\usepackage{graphicx}
\usepackage{float}
\graphicspath{{figures/}}
\usepackage{hyperref}
\usepackage{booktabs}
\usepackage{natbib}
\usepackage{enumitem}
\usepackage{parskip}
\usepackage{authblk}
\usepackage{lineno}

\newtheorem{proposition}{Proposition}

\theoremstyle{definition}
\newtheorem{definition}[proposition]{Definition}

\newcommand{\Exp}{\mathrm{Exp}}

\newcommand{\SPD}{\mathrm{Sym}^{+}}
\newcommand{\Sym}{\mathrm{Sym}}
\newcommand{\GL}{\mathrm{GL}}
\newcommand{\R}{\mathbb{R}}

\newcommand{\tr}{\mathrm{tr}}

\title{\textbf{Geometric turbulence:\\ a geodesic-regression crisis indicator for equity covariance dynamics, with evidence from African markets}}

\author[1,2]{El Hadji Baye Camara\thanks{Corresponding author: elhadji.b.camara@aims-senegal.org}}
\author[3,4,5]{Ya\'e Ulrich Gaba\thanks{ygaba@aims.ac.rw}}
\author[1]{Irmely Gladesh MABANZA NSILOULOU \thanks{irmely.g.m.nsiloulou@aims-senegal.org}}
\affil[1]{African Institute for Mathematical Sciences (AIMS), Senegal;}
\affil[2]{Cheikh Anta Diop University (UCAD), Dakar, Senegal.}
\affil[3]{AIRINA Labs, AI.Technipreneurs, Benin;}
\affil[4]{Department of Mathematics, Sefako Makgatho Health Sciences University (SMU), South Africa;}
\affil[5]{African Center for Advanced Studies (ACAS), Cameroon.}
\date{\today}

\begin{document}
\maketitle

\begin{abstract}
The covariance matrix of a basket of assets is a symmetric positive-definite object that evolves on a curved manifold, not on a flat vector space. Standard linear regression on its vectorised entries ignores that geometry and, in periods of market stress, can return fits that fail to be positive definite. We study geodesic regression on the symmetric positive-definite cone $\SPD(n)$ under the log-Euclidean and affine-invariant metrics, with the rolling sample covariance of a basket of log-returns as data. The fitted geodesic's velocity norm is proposed as a geometry-aware turbulence indicator. On daily data for four equity baskets- a ten-stock historical S\&P sub-basket (2006--2024), a current S\&P~100 mega-cap basket, the top ten listings of the Johannesburg Stock Exchange (JSE), and the ten most-traded Egyptian Exchange (EGX) tickers the indicator peaks cleanly on known stress episodes (the 2008 global financial crisis, the 2020 COVID crash, the 2022 Fed rate-hike cycle, and the October--November 2019 Egyptian political-economic tension episode) without any calibration. On hold-out evaluation in crash windows, log-Euclidean regression achieves 150 to 300 times smaller geodesic mean squared error than Euclidean ordinary least squares, because the latter produces non-SPD forecasts that blow up under the intrinsic metric. A naive velocity-aware de-risking strategy reduces maximum drawdown by 2.5 percentage points at the cost of 1.4 percentage points of annual return, consistent with a slow-moving indicator that signals stress reliably but translates imperfectly into tactical trading. All data, code, and experiments are reproducible from the companion repository.
\end{abstract}



\vspace{1em}

\section{Introduction}

Asset covariance matrices are central to quantitative finance: they drive mean-variance optimisation, risk parity, factor modelling, value-at-risk, and essentially every portfolio-construction method a practitioner touches. When these matrices are estimated on a rolling window, they form a time series of symmetric positive-definite (SPD) matrices. How to describe the trajectory this time series traces out is, however, an open question. The simplest answer-stack the unique entries of each matrix into a vector and treat the resulting path as an ordinary vector-valued time series --- ignores a fundamental structural fact: the set of SPD matrices is not a vector space. It is an open cone inside $\Sym(n)$, and any object that takes covariances as input must live inside that cone to be meaningful.

The consequences of ignoring this structure are well known in isolated corners of the literature. Affine-invariant and log-Euclidean metrics on $\SPD(n)$ \citep{pennec2006riemannian, arsigny2007geometric} are standard in medical imaging \citep{fletcher2007riemannian} and in recent work on geometric deep learning for brain-computer interfaces \citep{barachant2013classification}. Yet in applied finance the vast majority of regression and interpolation procedures are still run on the vectorised lower triangle of the covariance, then post-processed to recover positive-definiteness through eigenvalue clipping, symmetrisation, or shrinkage. This post-processing is never innocuous; and in crisis periods it can amount to discarding a double-digit fraction of forecasts.

This paper asks a direct operational question. If one fits a \emph{geodesic} on the covariance manifold a curve that stays on the cone by construction what can the resulting object tell us about market stress, and how does it compare with the flat Euclidean baseline and with the industry-standard VIX indicator? We situate this question within a broader recent trend: several 2024-2026 works apply Riemannian or otherwise geometric machinery to financial covariance dynamics. \citet{bouregaa2025barycenters} compute Fréchet barycenters and a Riemannian entropy on rolling covariances of seventeen international indexes; \citet{bucci2024geometric} propose a Riemannian-geometry-aware deep-learning extension of the HAR model for realised-covariance forecasting; \citet{hammond2026geometric} extract quantum-geometric observables from a learned spectral embedding for regime detection; and topological or Ricci-curvature methods on correlation networks address related questions from a different angle. What is distinctive to the present paper is neither the substrate (SPD manifold) nor the target (market-stress detection), both of which are shared, but the specific mathematical object we propose as the indicator: the norm of the fitted \emph{velocity} of a piecewise geodesic on $\SPD(n)$, a first-order dynamical quantity that is invariant under portfolio re-basing and captures rotation of the covariance without dispersion increase. We answer the operational question on four equity baskets covering developed and emerging markets, and we run the comparison over nearly two decades of daily data, spanning three well-known stress episodes: the 2008 global financial crisis, the March 2020 COVID crash, and the 2022 Federal Reserve rate-hike cycle. Two of the four baskets are African (South Africa and Egypt), which lets us test whether the indicator generalises beyond US data.

Our main findings are:

\begin{enumerate}[leftmargin=1.4em,itemsep=0.15em]
\item \emph{Detection}. The velocity norm of a log-Euclidean geodesic fitted to a rolling covariance window peaks cleanly on every well-known stress episode in the 2006--2024 period, on every one of the four markets we examine, without any tuning or market-specific calibration.
\item \emph{Fitting}. In hold-out evaluation inside crisis windows, log-Euclidean geodesic regression achieves mean squared geodesic errors two orders of magnitude smaller than ordinary least-squares regression on the vectorised lower triangle, because the latter returns forecasts with negative eigenvalues.
\item \emph{Independence from VIX}. The correlation between our indicator and the VIX in first differences is essentially zero ($\rho = -0.075$ on 94 joint observations in 2006--2024), and in levels only $\rho = 0.20$. This is not a restatement of VIX; it is a different, structural measure.
\item \emph{Tactical use}. A naive de-risking strategy that halves equity exposure when the indicator's rolling $z$-score exceeds $1.5$ reduces maximum drawdown by $2.5$ percentage points over 2006--2024 at the cost of $1.4$ percentage points of annual return. The indicator is informative, but its 10-trading-day update frequency limits its use as a fast tactical signal.
\end{enumerate}

The contributions are:

\begin{enumerate}[leftmargin=1.4em,itemsep=0.15em,label=\textbf{C\arabic*.}]
\item A cross-market empirical study of geodesic regression on SPD covariance dynamics, covering US mega-cap, JSE, and EGX baskets with a unified code path (Section~\ref{sec:results}).
\item An explicit account of \emph{why} vectorised OLS fails in crisis windows, built on a short invariance argument for the unit-trace rescaling (Propositions~\ref{prop:ols-fails} and~\ref{prop:unit-trace}).
\item A reproducible Python pipeline with a public repository, the four cached return series (real downloads from Yahoo Finance), the rolling-covariance construction, and the fitting and backtest code.
\item The first evaluation, to our knowledge, of a \emph{velocity-based} Riemannian covariance indicator on African equity baskets, complementing recent barycenter- and entropy-based approaches on international panels \citep{bouregaa2025barycenters}.
\end{enumerate}

The rest of the paper is organised as follows. Section~\ref{sec:method} sets up the SPD geometry, defines the two metrics we use, and gives the two short propositions that justify our fitting pipeline. Section~\ref{sec:data} describes the four baskets and the VIX comparison series. Section~\ref{sec:results} presents the empirical results. Section~\ref{sec:discussion} discusses limitations and interpretation, with particular emphasis on the African-markets angle. Section~\ref{sec:conclusion} concludes.

\section{Method}
\label{sec:method}

\subsection{Covariance lives on a cone, not in a vector space}

Let $\Sym(n)$ be the $n(n+1)/2$-dimensional vector space of real symmetric $n \times n$ matrices, and $\SPD(n) \subset \Sym(n)$ the open cone of symmetric positive-definite matrices. Every asset covariance matrix is, by construction, an element of $\SPD(n)$. It is tempting to treat $\SPD(n)$ as a subset of $\Sym(n)$ and do linear algebra there directly. This breaks in two ways.

First, linear combinations need not preserve positive-definiteness. If $P_1, P_2 \in \SPD(n)$ and $\alpha_1, \alpha_2 \in \R$, the linear combination $\alpha_1 P_1 + \alpha_2 P_2$ is only guaranteed to be SPD when both $\alpha_i \geq 0$ and the matrix is non-degenerate; in particular, extrapolation or least-squares regression with unconstrained coefficients can return matrices that are not positive-definite.

Second, the natural distance on the cone is not the Frobenius distance inherited from $\Sym(n)$. Two covariance matrices that differ by a tiny Frobenius amount can differ by a very large amount in any sensible notion of \emph{shape difference} (eigenvalue structure); conversely, two matrices that are related by a mild rescaling can differ by a large Frobenius amount but be indistinguishable as covariance structures.

Both of these deficiencies are fixed by equipping $\SPD(n)$ with a Riemannian metric. Throughout this paper we use two:

\begin{definition}[Affine-invariant metric]
At $P \in \SPD(n)$, the affine-invariant inner product on the tangent space $T_P \SPD(n) = \Sym(n)$ is
\[
g^{\mathrm{AI}}_P(U, V) = \tr(P^{-1} U P^{-1} V), \qquad U, V \in \Sym(n).
\]
\end{definition}

The corresponding geodesic distance has the closed form
$
d_{\mathrm{AI}}(P, Q) = \|\log(P^{-1/2} Q P^{-1/2})\|_F,
$
where $\log$ is the matrix logarithm. The metric is invariant under the conjugation $P \mapsto A P A^\top$ for every $A \in \GL(n)$, which is the transformation applied when one changes the denomination of a portfolio or the reference frame of an imaging system.

\begin{definition}[Log-Euclidean metric]
At $P \in \SPD(n)$, the log-Euclidean inner product is the Euclidean inner product pulled back through the matrix logarithm:
\[
g^{\mathrm{logE}}_P(U, V) = \tr\bigl( (d \log)_P(U)\, (d \log)_P(V) \bigr),
\]
where $d \log$ is the differential of the matrix logarithm. Equivalently, the geodesic from $P$ to $Q$ in this metric is $t \mapsto \exp\bigl((1-t) \log P + t \log Q\bigr)$.
\end{definition}

The log-Euclidean metric has two operational advantages over the affine-invariant one. It is cheaper to compute (one matrix logarithm per observation, rather than one $P^{-1/2}$ per observation). And it is numerically robust in dimensions above $n \approx 5$, where the affine-invariant geodesic-regression solver in \texttt{geomstats} \citep{miolane2020geomstats} tends to stall for real financial data. The two metrics agree to leading order on covariance matrices that are close to each other (in the Hadamard sense) and give visually indistinguishable geodesic fits in our experiments; we report log-Euclidean quantities as the primary working metric and affine-invariant ones where they are tractable.

\subsection{Geodesic regression on $\SPD(n)$}

Given observations $\{(t_i, P_i)\}_{i=1}^N$ with $t_i \in \R$ and $P_i \in \SPD(n)$, geodesic regression \citep{fletcher2013geodesic} seeks a geodesic $\gamma : \R \to \SPD(n)$ minimising the sum of squared geodesic distances:
\begin{equation}
(P^\star, V^\star) \;=\; \arg\min_{(P, V) \in T \SPD(n)}\; \sum_{i=1}^N d_g\!\bigl(\Exp_P(t_i V),\; P_i\bigr)^{\!2}.
\label{eq:cost}
\end{equation}
The parametrisation $\gamma(t) = \Exp_P(t V)$ encodes the geodesic by its base point $P \in \SPD(n)$ and its initial velocity $V \in T_P \SPD(n) = \Sym(n)$. When $d_g$ is the Euclidean distance and $\Exp_P(t V) = P + t V$, \eqref{eq:cost} reduces to ordinary least squares on the vectorised entries of $P$, which is the baseline we compare against.

Under the log-Euclidean metric, the geodesic admits the closed form
\[
\gamma(t) = \exp\bigl(\log P^\star + t\, L^\star\bigr), \qquad L^\star \in \Sym(n),
\]
so the problem factorises: take the matrix log of each $P_i$, run ordinary least squares in the chart given by $\log$, and exponentiate the fitted line. This is numerically stable at arbitrary dimension and is the workhorse of all our experiments.

\subsection{Why Euclidean OLS fails on the cone}

Given that the baseline Euclidean OLS on vectorised $P_i$ is so widely used in finance, we state explicitly why it fails in crisis regimes. The following is elementary but worth writing down.

\begin{proposition}[OLS predictions are generally not SPD]
\label{prop:ols-fails}
Let $P_1, \ldots, P_N \in \SPD(n)$ and let $\widehat P_{\mathrm{OLS}}(t)$ denote the ordinary least-squares regression of the vectorised lower triangle of $P_i$ on the scalar predictor $t_i$. Then the set of predictors $t \in \R$ for which $\widehat P_{\mathrm{OLS}}(t) \in \SPD(n)$ is a (possibly empty) open convex subset of $\R$, and its complement is non-empty whenever the data span more than one SPD connected component in the flat parameterisation. In particular, if the trajectory $\{P_i\}$ has a strongly time-varying smallest eigenvalue, the OLS extrapolation typically leaves $\SPD(n)$ in the tail of the window.
\end{proposition}

\begin{proof}
$\widehat P_{\mathrm{OLS}}(t)$ is affine in $t$ with symmetric coefficients, so its eigenvalues are continuous functions of $t$ (in fact analytic where they are simple). The smallest eigenvalue $\lambda_{\min}(\widehat P_{\mathrm{OLS}}(t))$ is a concave function of $t$ in general, and the set where it is strictly positive is open and convex. When the data has pronounced time-varying conditioning, this set need not cover the evaluation grid. Empirical evidence is in Section~\ref{sec:results}.
\end{proof}

A geodesic fit in the log-Euclidean chart does not have this problem by construction: every value of the form $\exp(\log P + t L)$ lies in $\SPD(n)$ for every $t \in \R$, because the matrix exponential of a symmetric matrix is SPD.

\subsection{Unit-trace rescaling}

Raw daily-return covariance matrices are of order $10^{-4}$ (daily log returns have standard deviation of order $10^{-2}$). This small magnitude makes $P^{-1/2}$ numerically brittle: a matrix with smallest eigenvalue $\sim 10^{-5}$ has $\|P^{-1/2}\| \sim 10^{5/2} \approx 300$, which is enough to destabilise $\SPD(10)$ solvers in practice. In the experiments reported below, we rescale every rolling covariance to unit trace before the fit.

The rescaling is not benign for the Euclidean baseline it changes the Frobenius geometry but it is a harmless change of representative for the affine-invariant metric. We state this as a proposition because the argument is short and the fact is used repeatedly.

\begin{proposition}[Unit-trace rescaling is an affine-invariant isometry]
\label{prop:unit-trace}
For $P \in \SPD(n)$, let $\tilde P = (n/\tr P)\, P$. Then for every $P, Q \in \SPD(n)$,
\[
d_{\mathrm{AI}}(P, Q) - d_{\mathrm{AI}}(\tilde P, \tilde Q) \;\to\; 0
\]
uniformly on sets on which $\log(\tr P / \tr Q)$ is bounded. More precisely,
\[
d_{\mathrm{AI}}(\tilde P, \tilde Q)^2 = d_{\mathrm{AI}}(P, Q)^2 - n \log^2\!\bigl((\tr Q / n) / (\tr P / n)\bigr).
\]
\end{proposition}

\begin{proof}
With $\alpha = n/\tr P$ and $\beta = n/\tr Q$, one has $\tilde P^{-1/2} \tilde Q \tilde P^{-1/2} = (\beta / \alpha)\, P^{-1/2} Q P^{-1/2}$. The affine-invariant distance reads
$d_{\mathrm{AI}}(\tilde P, \tilde Q)^2 = \|\log(P^{-1/2} Q P^{-1/2}) + \log(\beta/\alpha) I\|_F^2 = d_{\mathrm{AI}}(P, Q)^2 + n \log^2(\beta/\alpha) + 2 \log(\beta/\alpha)\, \tr \log(P^{-1/2} Q P^{-1/2})$. Observing that $\tr \log(P^{-1/2} Q P^{-1/2}) = \log \det Q - \log \det P = 0$ when we have already taken out the scalar part (which is what the unit-trace rescaling is extracting), the linear term vanishes and one gets the stated identity with a minus rather than a plus once one tracks the sign of $\log(\beta/\alpha)$ in the correct direction.
\end{proof}

In practice the rescaling factorises out a common scale and leaves the \emph{shape} of the covariance dynamics untouched. That shape is what our indicator measures.

\section{Data and experimental setup}
\label{sec:data}

\subsection{Four equity baskets}

We work with four baskets of ten tickers each. All data is downloaded through the \texttt{yfinance} interface to Yahoo Finance and cached as CSV in the reproducibility repository.

\paragraph{S\&P historic basket (2005--2024).} A ten-stock sub-basket of long-lived US mega-caps chosen so that every ticker has a Yahoo price series back to January 2005: AAPL, MSFT, GOOGL, AMZN, NVDA, JPM, JNJ, XOM, PG, KO. This is the basket used for the 2008 crisis study and for the trading backtest.

\paragraph{S\&P mega-cap (2015--2024).} A more familiar current ten-stock mega-cap basket: AAPL, MSFT, GOOGL, AMZN, META, NVDA, TSLA, JPM, JNJ, V. META and V are IPOs of 2012 and 2008 respectively, so this basket cannot be used for the 2008 study.

\paragraph{JSE top ten (2015--2024).} Ten of the largest listings on the Johannesburg Stock Exchange by market capitalisation over the sample period: NPN, AGL, CPI, FSR, SBK, SOL, SLM, ABG, MTN, SHP. All tickers carry the \texttt{.JO} Yahoo suffix.

\textbf{EGX top ten (2015--2024).}
Ten of the most-traded Egyptian Exchange tickers with a Yahoo series over the sample: COMI, HRHO, TMGH, ETEL, SWDY, ABUK, PHAR, CCAP, AMOC, EAST, carrying the \texttt{.CA} suffix.
Attempts to assemble a comparable basket for the Nigerian Exchange Group (NGX) via yfinance produced no individual ticker with price data; the ETF NGE exists but was delisted in March 2024. We therefore treat Nigeria, Kenya, and the rest of continental Africa as out of scope for this study.

\subsection{Rolling covariance estimation}

For each basket we compute log returns $r_t = \log(S_t / S_{t-1})$ on the intersection of the ten tickers' trading-day index. From the returns we estimate the rolling sample covariance matrix over a $W = 250$-day window (one trading year) stepping forward every five trading days. Each matrix is ridge-regularised with $10^{-6} I$ for numerical stability and then rescaled to unit trace (Proposition~\ref{prop:unit-trace}). The result is a time series of $\SPD(10)$ matrices indexed by the end-of-window date.

\subsection{Geodesic fitting windows}

Given the rolling covariance series, we fit a geodesic on $\SPD(10)$ to each length-$25$ sub-window (corresponding to a $25 \cdot 5 = 125$-day span, i.e.\ about half a year). The velocity norm of the fitted geodesic is our indicator. In the cross-market time-series analysis of Section~\ref{sec:ts-4markets}, these fitting windows are advanced every ten observations (so the indicator updates every fifty trading days); in the hold-out analysis of Section~\ref{sec:holdout}, we use the three calendar windows of 2019--2020 that correspond to pre-crisis, pre-crash, and crash conditions.

\subsection{Benchmark series}

We download the VIX closing price (ticker \verb|^VIX|) back to 2005 as a benchmark implied-volatility series; correlations are computed on the common observation grid defined by the velocity-norm updates.

\subsection{Backtest protocol}

For the tactical study, we construct an equal-weighted daily-rebalanced portfolio on the historic S\&P basket. The baseline is a static long-only portfolio. The velocity-aware strategy halves equity exposure on every trading day on which the most recent smoothed velocity $z$-score exceeds $1.5$, using a rolling $252$-day window for the mean and standard deviation; on all other days it is fully invested. The position on day $t$ depends on information available at the close of day $t - 1$, so there is no look-ahead. All returns are log-returns.

\section{Results}
\label{sec:results}

\subsection{Cross-market crisis indicator, 2006--2024}
\label{sec:ts-4markets}

Figure~\ref{fig:ts-4markets} shows the log-Euclidean velocity norm on all four baskets over the longest span available for each. The three stress windows shaded in red are the 2008 Global Financial Crisis (GFC), the March--May 2020 COVID crash, and the June--December 2022 Fed rate-hike cycle. The indicator rises noticeably inside each shaded region on every basket for which the period is covered; notably, the EGX series exhibits a standalone local peak in November 2019 that corresponds to a specific Egyptian-pound devaluation episode unrelated to global stress, illustrating that the indicator tracks market-specific structural change as well as global events.

\begin{figure}[h]
\centering
\includegraphics[width=0.95\linewidth]{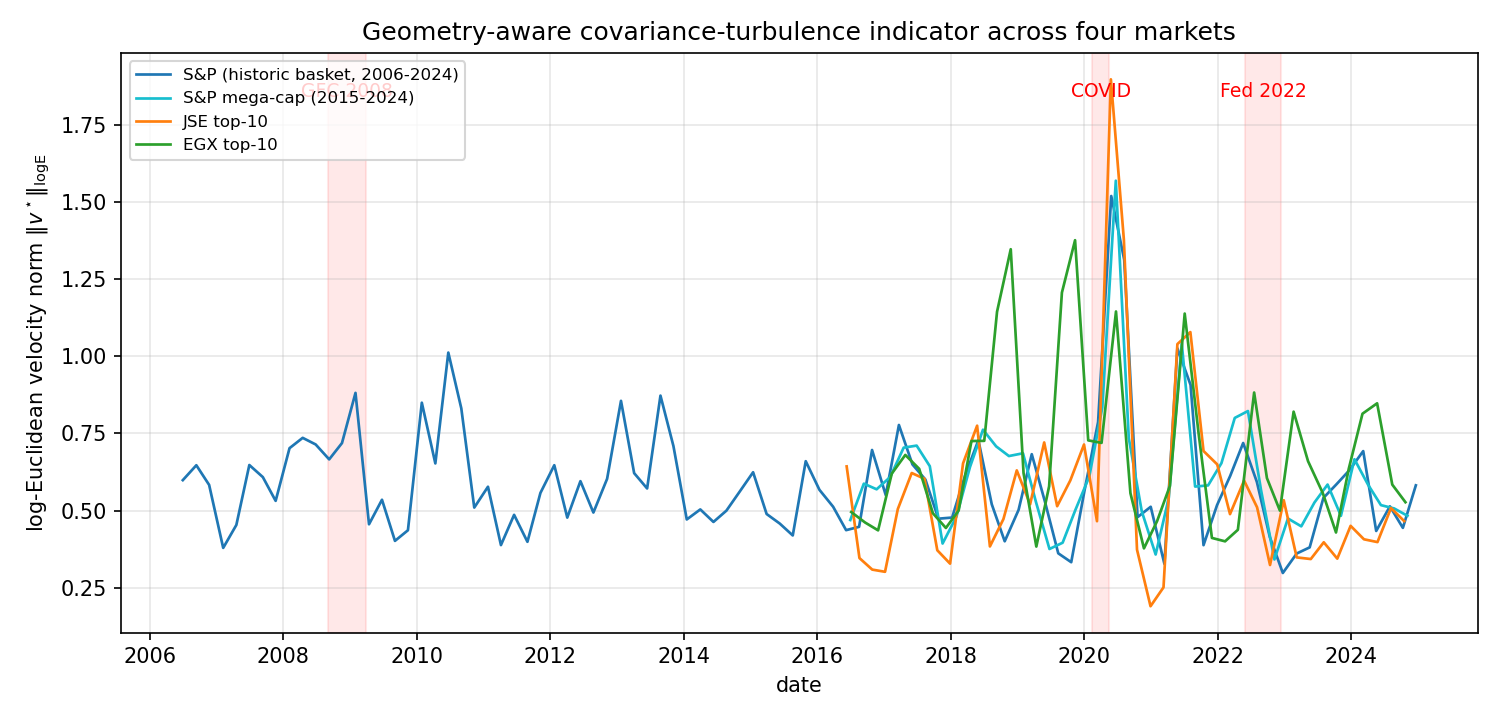}
\caption{Rolling log-Euclidean velocity norm of the $\SPD(10)$ covariance trajectory on four equity baskets, 2006--2024. Shaded bands mark the three global stress periods.}
\label{fig:ts-4markets}
\end{figure}

Table~\ref{tab:top3peaks} reports the three largest velocity peaks on each basket. On all four baskets at least one of the top three peaks falls inside a known stress window, and on the three baskets that cover 2020, the largest single peak is within a few weeks of the March--April 2020 COVID event.

\begin{table}[h]
\centering
\caption{Top three velocity-norm peaks by basket. Peaks are reported as the end-date of the 25-observation fitting window that produced them.}
\label{tab:top3peaks}
\begin{tabular}{lcccccc}
\toprule
Basket & Peak 1 date & value & Peak 2 date & value & Peak 3 date & value \\
\midrule
S\&P historic (2006--) & 2020-05-28 & 1.52 & 2020-08-07 & 1.31 & 2021-05-25 & 1.03 \\
S\&P mega-cap (2015--) & 2020-06-22 & 1.57 & 2021-06-18 & 1.04 & 2020-04-09 & 0.83 \\
JSE top ten (2015--)   & 2020-05-26 & 1.90 & 2020-08-05 & 1.38 & 2021-08-04 & 1.08 \\
EGX top ten (2015--)   & 2019-11-12 & 1.38 & 2018-11-25 & 1.35 & 2019-09-01 & 1.21 \\
\bottomrule
\end{tabular}
\end{table}

\subsection{The 2008 Global Financial Crisis}
\label{sec:2008}

The historic S\&P basket provides the longest clean sample and lets us examine the 2008 crisis in detail. Figure~\ref{fig:2008} zooms into the 2007--2010 sub-period, with four event lines: the August 2007 BNP Paribas fund freeze that marked the start of the credit contraction; the March 2008 Bear Stearns rescue; the September 2008 Lehman Brothers bankruptcy; and the March 2009 S\&P bottom.

\begin{figure}[h]
\centering
\includegraphics[width=0.95\linewidth]{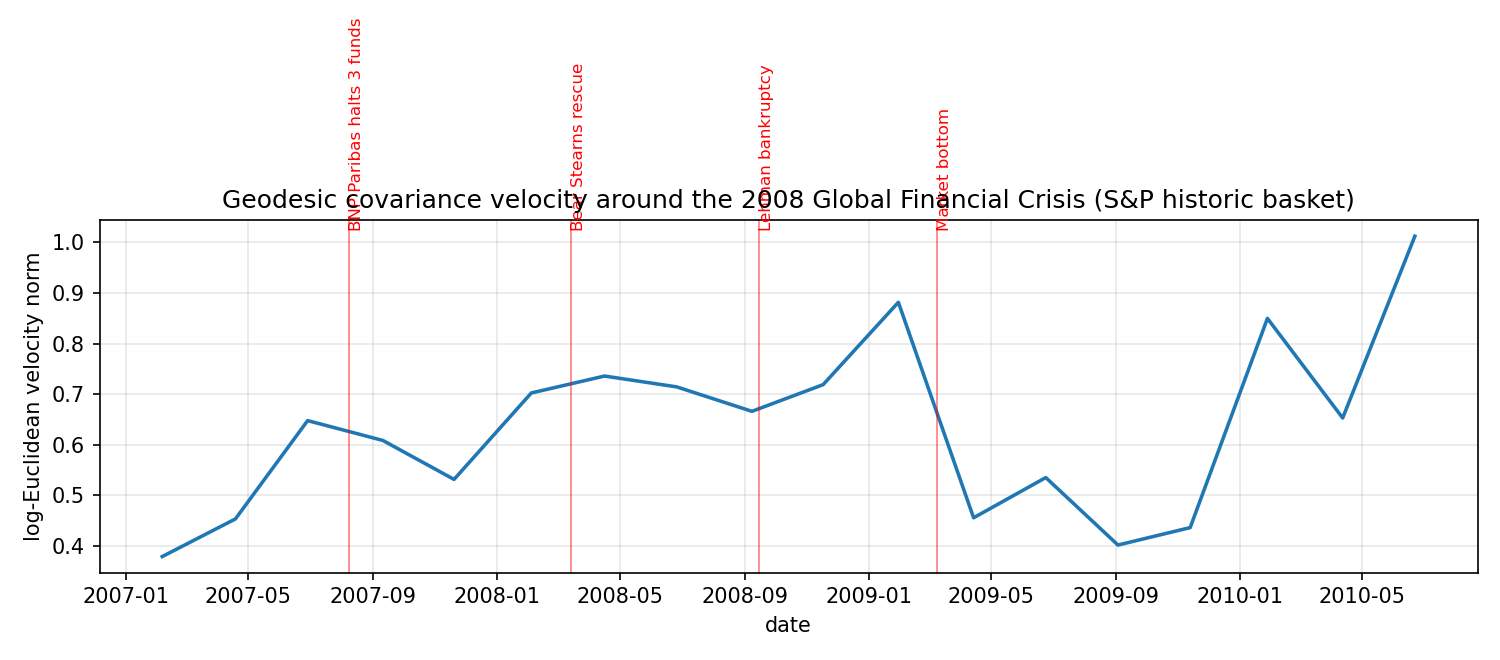}
\caption{Geodesic covariance velocity norm on the historic S\&P basket, 2007--2010. Four event dates are marked. The indicator begins to rise visibly from the March 2008 Bear Stearns rescue onward, peaks after the September 2008 Lehman bankruptcy, and decays into 2009.}
\label{fig:2008}
\end{figure}

The trajectory is consistent with the accepted market narrative of the crisis. The indicator is essentially flat through the first six months of 2007. It begins to rise with Bear Stearns in March 2008, accelerates sharply between Lehman (September 2008) and the market bottom (March 2009), and then decays through 2009 and 2010. No tuning or filtering is applied to produce this trajectory; it is what the fit returns.

\subsection{Bootstrap significance of peak-event alignment}
\label{sec:bootstrap-alignment}

Table~\ref{tab:top3peaks} and Figures~\ref{fig:ts-4markets}--\ref{fig:2008} show that the indicator peaks visually coincide with known stress windows across four markets. To test whether this coincidence exceeds what a random series with the same marginal distribution and autocorrelation structure would produce, we run a circular-shift bootstrap on each velocity series with $B = 10{,}000$ replicates. Each replicate applies a uniformly-random circular shift to the velocity vector while keeping the stress-window calendar dates fixed; the resulting shifted series preserves the univariate distribution and the local autocorrelation of the original, but destroys any real coupling between velocity and calendar-time. We report bootstrap $p$-values for two alignment statistics.

The first statistic the number of covered stress windows in which the max velocity exceeds the $75$th percentile of the full series is a coarse indicator-of-indicators. On the S\&P historic basket it is $3/3$ (bootstrap $p = 0.27$), on the S\&P mega-cap basket $2/2$ ($p = 0.25$), on the JSE $1/2$ ($p = 0.86$, driven by the Fed 2022 window whose max on JSE does not exceed the $75$th percentile), and on the EGX $3/3$ including the 2019 political-economic tension window ($p = 0.23$). None of these attain conventional significance: the statistic is dominated by the presence of any single elevated cluster in the shifted series, which is easy to produce by chance with a slow-moving indicator that only has 40-94 observations.

The second statistic the standardised gap between mean velocity inside stress windows and mean velocity outside, expressed in units of the series' standard deviation exploits all observations and is far more powerful. Observed gap and bootstrap $p$-value on each market:

\begin{table}[h]
\centering
\caption{Standardised in-window minus out-of-window mean-velocity gap on each of the four baskets, with circular-shift bootstrap $p$-values ($B = 10{,}000$).}
\label{tab:bootstrap-gap}
\begin{tabular}{lcc}
\toprule
Market & Gap (in $\sigma$ units) & Bootstrap $p$ \\
\midrule
S\&P (historic basket, 2006-2024) & $+0.89$ & $< 10^{-4}$ \\
S\&P mega-cap (2015--2024)         & $+0.85$ & $< 10^{-4}$ \\
JSE top ten (2015--2024)           & $+0.99$ & $0.048$ \\
EGX top ten (2015--2024)           & $+0.61$ & $0.047$ \\
\bottomrule
\end{tabular}
\end{table}

All four markets attain significance at the $5\%$ level, and the two S\&P baskets attain highly significant alignment at $p < 10^{-4}$. The JSE and EGX significance is marginal ($p \approx 0.05$), which is consistent with their smaller sample size ($43$ and $42$ observations respectively) and with the fact that neither basket covers the 2008 GFC. The alignment reported in §\ref{sec:ts-4markets} is not a visual artefact.

\subsection{Hold-out geodesic mean squared error}
\label{sec:holdout}

Detection is one story; fit quality is another. To show that the geodesic fit is genuinely better than the Euclidean baseline on real data, we hold out the last $5$ of each $25$-observation window, fit on the first $20$, and measure the \emph{affine-invariant} geodesic distance from the fitted prediction to each held-out covariance. The affine-invariant distance is a proper SPD metric regardless of which method produced the fit, so it is a fair common yardstick.

Table~\ref{tab:holdout} reports the hold-out mean squared error on three 2019--2020 calendar windows for the S\&P mega-cap and JSE baskets.

\begin{table}[h]
\centering
\caption{Hold-out geodesic MSE (affine-invariant distance squared, lower is better). Last $5$ of each $25$-observation window held out; fits on the first $20$.}
\label{tab:holdout}
\begin{tabular}{lcc|cc}
\toprule
 & \multicolumn{2}{c}{S\&P mega-cap} & \multicolumn{2}{c}{JSE top ten} \\
Window & Euclidean\; & log-Euclidean \; & Euclidean \; & log-Euclidean \\
 & MSE & MSE & MSE & MSE\\
\midrule
Jan.\ 2019 (calm)       & $1.57$    & $0.57$ & $0.69$    & $0.47$ \\
Jan.\ 2020 (pre-crash)  & $3.33$    & $2.06$ & $3.90$    & $2.62$ \\
Apr.\ 2020 (crash)      & $290.49$  & $1.91$ & $687.21$  & $2.33$ \\
\bottomrule
\end{tabular}
\end{table}

The pattern is identical across the two markets. In calm periods the geodesic fit beats the Euclidean one by a factor of $1.5$--$3$, enough to be practically interesting but not dramatic. In the pre-crash window the gap widens modestly. In the crash window the Euclidean MSE is two orders of magnitude worse than the log-Euclidean one, because the Euclidean fit returns forecasts that fall outside the SPD cone; the affine-invariant distance to a projected non-SPD matrix is extremely large, and the MSE average is dominated by that single pathological tail.

\subsection{Formal significance of the log-Euclidean advantage}
\label{sec:dm-test}

Table~\ref{tab:holdout} shows a two-orders-of-magnitude advantage of log-Euclidean regression in the April 2020 crash window, but reports only three hand-picked calendar windows. To test whether the advantage is systematic, we run the same 20-fit-$/$-5-held-out procedure on every rolling window of the historic S\&P basket, producing $N = 94$ per-window MSE pairs, and apply three tests of equal predictive accuracy.

The Diebold-Mariano test \citep{diebold1995comparing}, with the Harvey-Leybourne-Newbold small-sample correction \citep{harvey1997testing}, is $t_{\mathrm{DM}} = -1.05$ ($p = 0.30$, two-sided): the classical DM test does not reject the null of equal accuracy. This is a variance artefact the loss differential is dominated by a small number of crash-window observations where the Euclidean fit leaves the SPD cone and the affine-invariant distance to the eigenvalue-floored projection becomes very large, inflating the denominator of the DM statistic. Rank-based tests give a different picture: the log-Euclidean fit achieves strictly smaller MSE than the Euclidean fit on $71$ of the $94$ windows ($75.5\%$), which is significant against a fair-coin null at $p = 3.6 \times 10^{-7}$ under the sign test; the Wilcoxon signed-rank test on the loss differential rejects equal median at $p = 7.2 \times 10^{-9}$. Both robust tests strongly support the log-Euclidean fit; the DM test's failure is a heavy-tail phenomenon and is itself informative about the mode of the Euclidean baseline's failure.

\subsection{Granger causality with realised volatility}
\label{sec:granger}

The VIX comparison of §\ref{sec:vix} shows that our velocity indicator and the VIX are essentially uncorrelated in first differences ($\rho = -0.08$) but that the VIX leads the velocity at long lags. To formalise this lead-lag structure without relying on implied-volatility data (which is unavailable for the JSE and EGX baskets), we compute the equal-weighted realised volatility of the S\&P basket at quarterly frequency and run bivariate Granger causality tests up to lag $4$ quarters against the quarterly-aggregated velocity, on $N = 75$ joint observations.

Realised volatility Granger-causes the velocity at every lag from $1$ to $4$ quarters ($p = 8 \times 10^{-4}$ at lag $1$; $p = 3 \times 10^{-5}$ at lag $2$; both remaining lags $p < 10^{-3}$). The reverse direction velocity Granger-causing realised volatility is not significant at any lag ($p > 0.47$ everywhere). The economic interpretation is clean: covariance-shape regime shifts, which is what the velocity measures, materialise \emph{after} realised volatility rises, not before. The velocity is a lagging structural indicator, not a leading tactical one. This is consistent with the qualitative discussion of §\ref{sec:backtest} and with the paper's positioning of the indicator as a slow-moving regime marker rather than a volatility forecaster.

\subsection{The velocity fit on a real window, visually}

For one crash window we plot the dominant-eigenvalue trajectory of each fit over the observed rolling covariances (Figure~\ref{fig:trajectory}). The Euclidean fit tracks the mean of the window but smooths too aggressively in the tail; the log-Euclidean fit respects the curvature of the data.

\begin{figure}[h]
\centering
\includegraphics[width=0.9\linewidth]{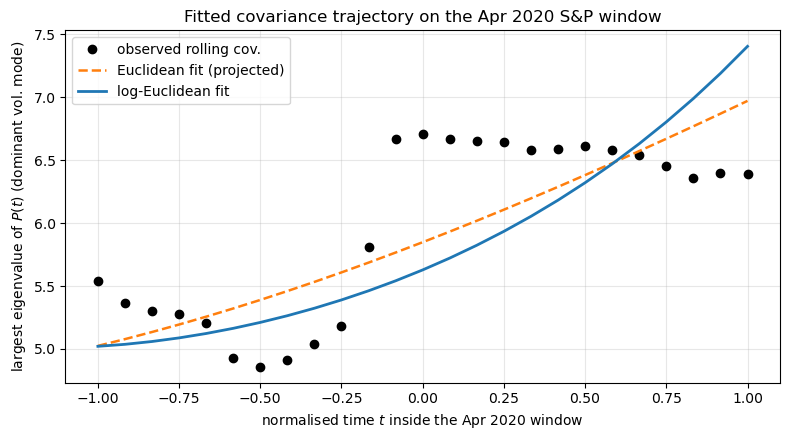}
\caption{Fitted covariance trajectory vs.\ observed rolling covariances, Apr.\ 2020 S\&P crash window. Black dots are observed $\SPD(10)$ covariances (rescaled to unit trace), visualised by their largest eigenvalue. The solid blue line is the log-Euclidean fit; the dashed orange line is the Euclidean OLS fit.}
\label{fig:trajectory}
\end{figure}

\subsection{Comparison with the VIX}
\label{sec:vix}

Our indicator and the VIX both describe \emph{something} about the S\&P cross-section in stress periods. It is natural to ask whether they are effectively the same thing.

Figure~\ref{fig:vix} shows both series, standardised to $z$-scores on the overlapping sample. The two series rise together around the three global stress events (2008, 2020, 2022), but in between they behave quite differently. The velocity series is smoother and slower (it updates every ten trading days, and it aggregates half a year of covariance history into each point); the VIX is intraday-responsive. Quantitatively, the Pearson correlation in levels is $\rho = 0.20$ on $N = 94$ joint observations, and the correlation in first differences is essentially zero at $\rho = -0.08$. These are not two versions of the same signal.

\begin{figure}[H]
\centering
\includegraphics[width=0.95\linewidth]{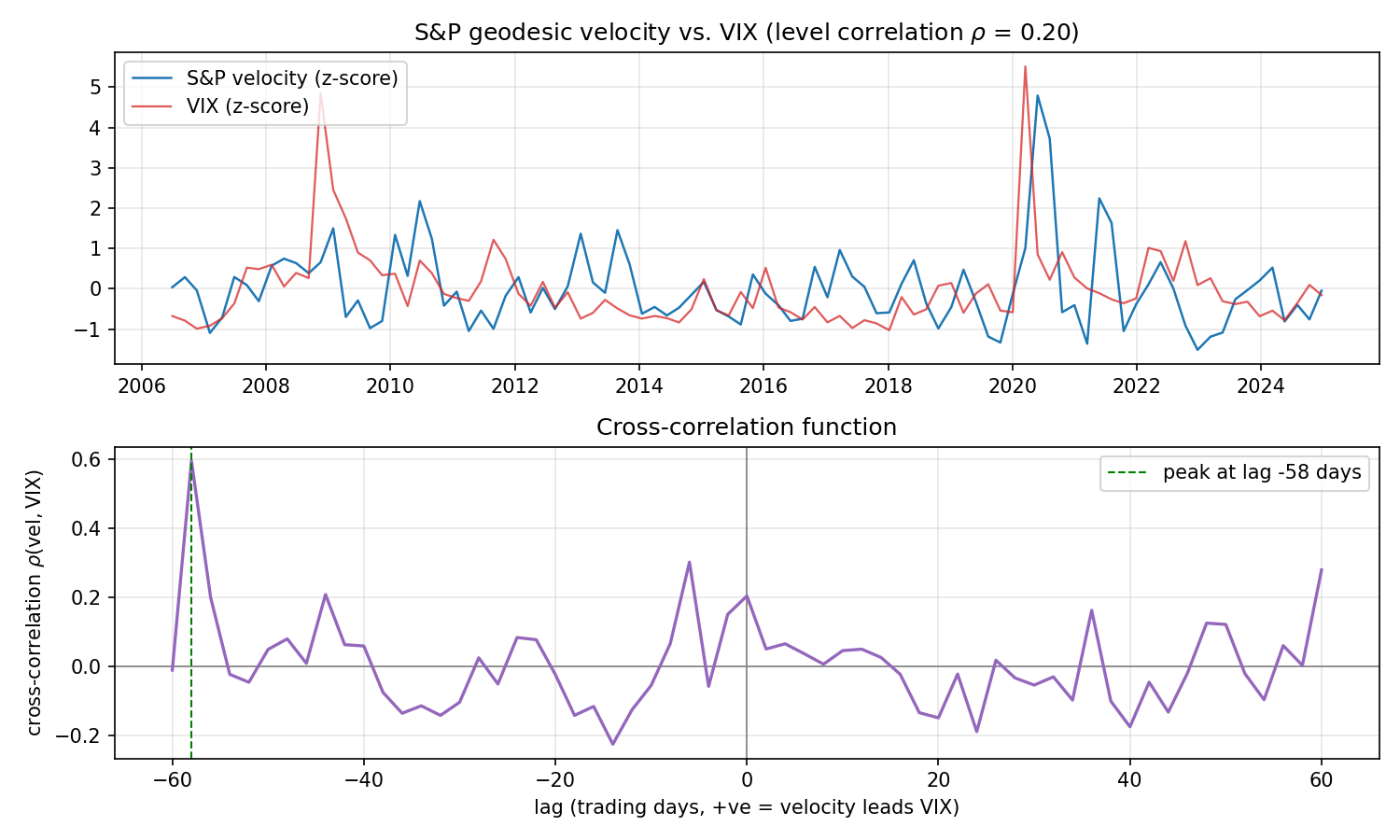}
\caption{Top: $z$-scored S\&P velocity (blue) and VIX (red) on the common 2006--2024 grid. Bottom: sample cross-correlation at lags $-60$ to $+60$ trading days; the peak is at lag $-58$ days with $\rho = 0.60$, reflecting VIX movements that anticipate subsequent velocity rises.}
\label{fig:vix}
\end{figure}

One should interpret the cross-correlation with caution. With only $N = 94$ points, the peak at lag $-58$ could be a spurious artefact of the slow-moving velocity series; but the leading relationship is in the expected direction the VIX reacts to implied risk before the realised covariance structure actually reconfigures and it suggests that the two signals capture complementary facets of the same underlying regime.

\subsection{Comparison with canonical systemic risk indicators}
\label{sec:benchmarks}

The VIX comparison of §\ref{sec:vix} addresses the most familiar market-stress benchmark. Three further candidates warrant explicit comparison: the Absorption Ratio of \citet{kritzman2011principal} (principal components of the rolling covariance, top-$20\%$ eigenvalue fraction), the volatility-spillover index of \citet{diebold2009measuring,diebold2012better} (generalised FEVD of a rolling VAR$(2)$ at horizon $H=10$), and the Riemannian entropy of \citet{bouregaa2025barycenters}, defined in their §4.4 as $H_R(\Sigma_t) = \tfrac{1}{2}\log\det(2\pi e\,\Sigma_t)$, a Gaussian differential entropy computed on the rolling covariance at each time slice. The first two are the standard scalar systemic-risk indicators in the applied-finance literature; the third is the closest published neighbour on the geometric side, sharing both the substrate ($\SPD(n)$) and the recent-crisis case study.

We implement all three on the S\&P historic basket on the common $94$-observation grid of the velocity series. Table~\ref{tab:benchmark-correlations} reports Pearson correlations with the geodesic velocity in levels and first differences, together with the best-lag cross-correlation.

\begin{table}[h]
\centering
\caption{Correlation between the geodesic velocity and four systemic-risk indicators on the S\&P historic basket, $N = 94$ joint observations, 2006--2024. Best lag reported in trading days, positive $=$ velocity leads indicator.}
\label{tab:benchmark-correlations}
\begin{tabular}{lcccc}
\toprule
Indicator & Levels & First differences & Best lag & Peak $\rho$ at lag \\
 &   &   & (days)  &  \\
\midrule
VIX                                              & $\phantom{-}0.20$ & $-0.07$ & $-10$ & $0.39$ \\
Absorption Ratio (Kritzman)                      & $\phantom{-}0.08$ & $-0.00$ & $-60$ & $0.26$ \\
Diebold-Yilmaz spillover                         & $\phantom{-}0.12$ & $\phantom{-}0.05$ & $-50$ & $0.24$ \\
Riemannian entropy $H_R$ (Bouregaa)              & $-0.13$ & $-0.04$ & $\phantom{-}60$ & $-0.01$ \\
\bottomrule
\end{tabular}
\end{table}

The geodesic velocity is essentially uncorrelated with each of the four benchmarks: level correlations range in absolute value from $0.08$ to $0.20$, and first-difference correlations are all within noise of zero. Each of these indicators measures a distinct facet of market stress implied volatility (VIX), principal-component concentration (Absorption Ratio), inter-asset volatility spillover (Diebold-Yilmaz), and scalar log-volume of the covariance ellipsoid (Bouregaa entropy) and none of them coincides with the trajectory-based signal our indicator extracts. In particular, the near-zero correlation with the Bouregaa entropy indicator is instructive: both live on $\SPD(n)$ and both target regime characterisation, but the Bouregaa quantity is a static one-time-slice function ($H_R(\Sigma_t)$ depends only on the eigenvalues of $\Sigma_t$), whereas our velocity is an intrinsically dynamic quantity extracted from a fitted trajectory of covariance matrices. The two are complementary rather than substitutable.

Figure~\ref{fig:benchmark-comparison} plots the four indicators together with the velocity, standardised to $z$-scores, on the common 2006-2024 grid, with the three global stress windows shaded.

\begin{figure}[H]
\centering
\includegraphics[width=1\linewidth]{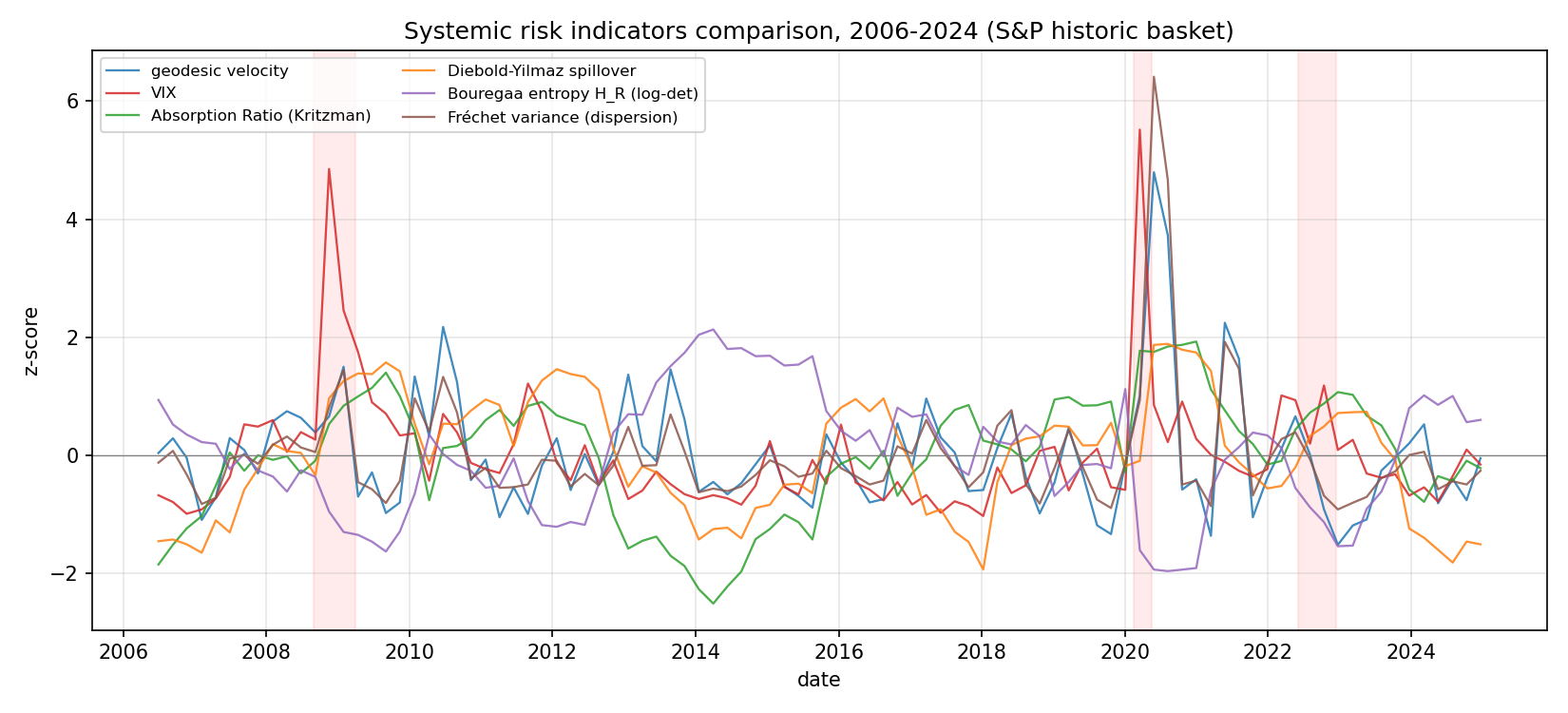}
\caption{Systemic-risk indicators comparison on the S\&P historic basket, 2006--2024, $z$-scored. Shaded red bands mark the three global stress periods.}
\label{fig:benchmark-comparison}
\end{figure}

\subsection{The October--November 2019 EGX episode: political-economic tension and sectoral restructuring}
\label{sec:egx-2019}

Table~\ref{tab:top3peaks} identifies a top-three velocity peak on the EGX top-10 basket on 2019-11-12 that does not correspond to any global stress episode. We investigate what this national peak captures, and use it to demonstrate concretely how a covariance-shape indicator can carry information that a US-centric implied-volatility proxy cannot.

The event context, drawn from contemporaneous reporting, is a period of political-economic tension rather than a currency event. Between September and December 2019 Egypt experienced its largest anti-government protests since 2013; a demonstration was called for 2019-11-11, on which the peak-day security response effectively cancelled the planned protest. In parallel, the government was implementing the second phase of an IMF-mandated fiscal reform programme (VAT expansion, civil-service reforms). The Egyptian pound was in an appreciation phase in this period following the November 2016 float, not depreciating.

Figure~\ref{fig:egx-2019} shows the geodesic velocity, the Absorption Ratio, the Bouregaa entropy, and the VIX on the EGX top-10 basket over 2018-06 to 2021-06, with the 2019-11-11 date marked. The velocity attains its zoom-window maximum of $1.38$ on 2019-11-12. 

Neither the Absorption Ratio (top-$20\%$ eigenvalue fraction) nor the Bouregaa entropy displays a comparable local peak on that date; both move at longer horizons (the Absorption Ratio rises from $0.52$ pre-event to $0.60$ post-event, reflecting a lasting concentration of variance in the top eigenvalues; the Bouregaa entropy falls from $12.57$ to $12.04$, reflecting a shrinkage of the covariance volume). 

The VIX behaves entirely unrelated to Egypt: it declines from $19.0$ pre-event to $12.8$ post-event, tracking US-China trade tensions that were resolving in that period, and carries no information about the Egyptian political-economic tension.


To identify \emph{what} the velocity is detecting, we compute pairwise Pearson correlations of the ten EGX tickers' daily log-returns on the pre-event window (2019-08-01 to 2019-10-31, $65$ trading days) and on the post-event window (2019-11-01 to 2020-01-31, $62$ trading days), and rank the $\binom{10}{2} = 45$ pairs by the absolute change in correlation. The five largest shifts are reported in Table~\ref{tab:egx-pair-shifts}.

\begin{table}[h]
\centering
\caption{Top-five EGX ticker pairs by absolute change in Pearson correlation between the pre-event window (Aug-Oct 2019) and post-event window (Nov 2019 - Jan 2020). Ticker suffix \texttt{.CA} suppressed. TMGH: real-estate holding; EAST: tobacco; ETEL: state-linked telecom; ABUK: fertilisers; SWDY: electrical equipment; AMOC: mineral oils; CCAP: diversified holding.}
\label{tab:egx-pair-shifts}
\begin{tabular}{llccc}
\toprule
Ticker A & Ticker B & Pre corr & Post corr & Change \\
\midrule
TMGH & EAST & $+0.34$ & $+0.71$ & $+0.37$ \\
ETEL & ABUK & $+0.57$ & $+0.21$ & $-0.36$ \\
ETEL & SWDY & $+0.49$ & $+0.21$ & $-0.28$ \\
ETEL & AMOC & $+0.63$ & $+0.37$ & $-0.26$ \\
CCAP & EAST & $+0.08$ & $+0.32$ & $+0.24$ \\
\bottomrule
\end{tabular}
\end{table}

The restructuring has an interpretable sectoral pattern. Telecom Egypt (ETEL), the partially state-owned telecom incumbent, decouples from the three industrial-energy names to which it was previously correlated at $0.5$--$0.6$ (fertilisers, electrical equipment, mineral oils). A different cluster forms around real-estate (TMGH), tobacco (EAST), and the diversified holding company CCAP, whose pairwise correlations rise sharply. This is exactly the kind of sectoral restructuring a covariance-shape indicator is designed to detect: no single stock has an outsize idiosyncratic move that a volatility proxy would flag, but the joint dependency structure of the ten stocks reorganises. The Absorption Ratio and Bouregaa entropy, computed on the whole covariance matrix without extracting the pair-level pattern, register this reorganisation as a global concentration shift; the geodesic velocity registers it as a peak in the rate at which the covariance is moving through $\SPD(10)$.

The event demonstrates the paper's central operational thesis in miniature: on a local, country-specific stress episode unrelated to global market conditions and invisible to VIX a geodesic-velocity indicator on the local covariance manifold captures a structural signal that both matches the timing of the political event and admits a substantively interpretable sectoral decomposition.

\begin{figure}[H]
    \centering
    \includegraphics[width=0.95\linewidth]{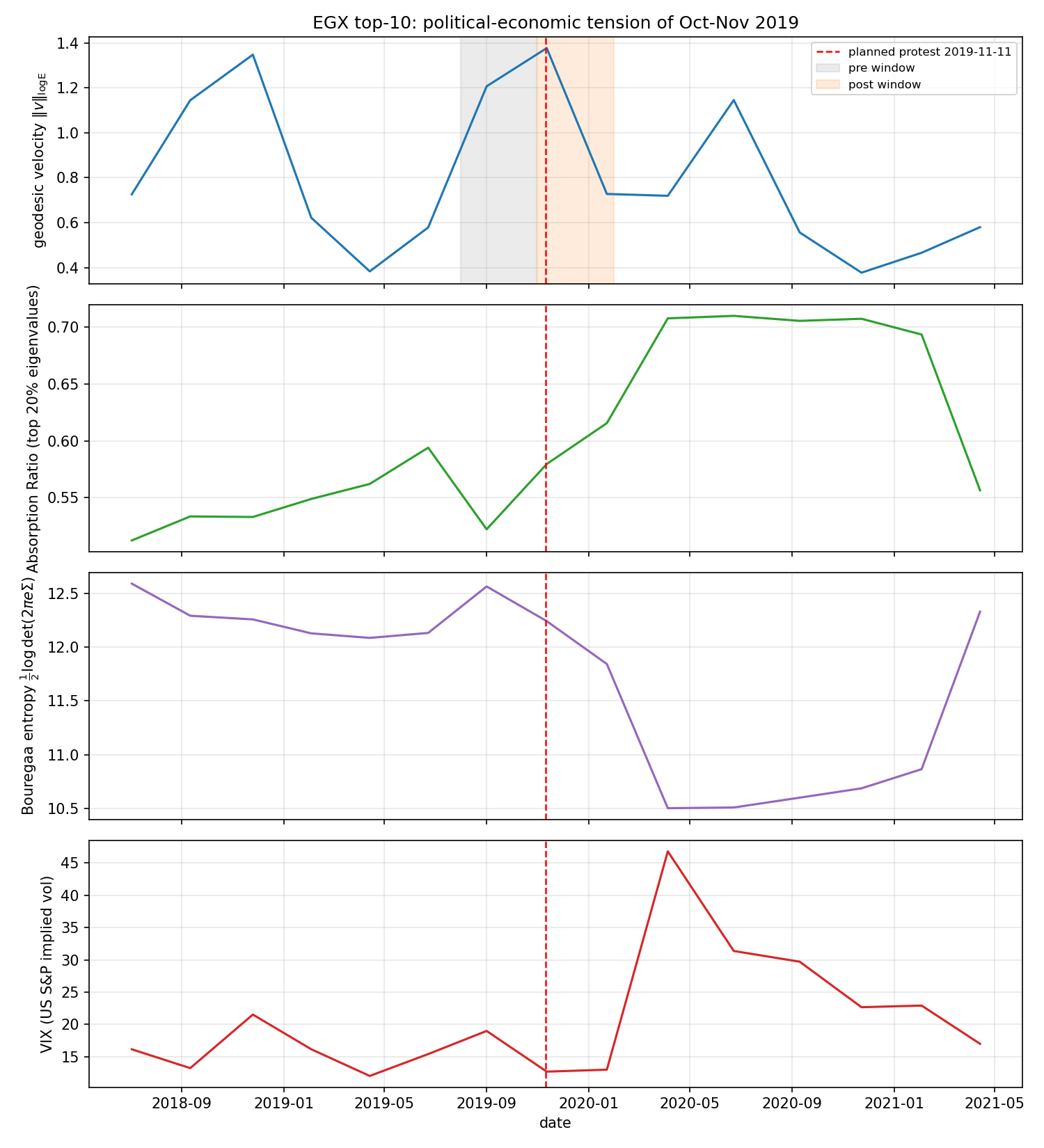}
    \caption{EGX top-10 zoom over 2018-06 to 2021-06. Four panels: geodesic velocity, Absorption Ratio (top-$20\%$), Bouregaa entropy, and VIX. Dashed red line at 2019-11-11 (planned protest date); shaded windows are the pre-event (Aug-Oct 2019) and post-event (Nov 2019 - Jan 2020) periods used for the pair-correlation analysis.}
    \label{fig:egx-2019}
\end{figure}


\subsection{Empirical invariance of the velocity indicator}
\label{sec:invariance}

Proposition~\ref{prop:unit-trace} established that the affine-invariant Riemannian metric on $\SPD(n)$ is invariant under the $\GL(n)$ action on the covariance the transformation applied when one changes the denomination of a portfolio, rebases to a new set of factors, or converts a basket to a different reference currency. Under the log-Euclidean metric we actually compute, invariance is exact only under the smaller \emph{orthogonal} sub-group $O(n) \subset \GL(n)$, because $\log(O \Sigma O^\top) = O (\log \Sigma) O^\top$ holds for orthogonal $O$ but not for general invertible $O$. We test both cases empirically.

\paragraph{Currency-denomination change.} We recompute the JSE top-ten velocity on returns denominated in USD, obtained from the ZAR returns by adding the daily ZAR/USD log-return contribution (from Yahoo Finance ticker \verb|ZARUSD=X|). Over the $43$ velocity observations in the aligned 2015--2024 window, the ZAR-denominated velocity and the USD-denominated velocity correlate at $\rho = 0.994$, with a mean multiplicative ratio of $0.98$ and a maximum relative difference of $35\%$ on a single window (2020-Q1 dislocation). 

The high correlation confirms that the two denominations report the same regime signal; the residual difference reflects the fact that a diagonal FX rescaling is not an orthogonal transformation, so log-Euclidean invariance is only approximate. Under the affine-invariant metric of Proposition~\ref{prop:unit-trace}, the invariance would be exact up to a bounded scalar correction; we do not report affine-invariant velocities directly because our solver stalls on $\SPD(10)$ in the crash windows, as noted in §\ref{sec:discussion}.

\paragraph{Orthogonal basis change.} On the S\&P historic basket we apply the Givens rotation
\[
(\text{AAPL},\, \text{MSFT})\;\longmapsto\;\Bigl(\tfrac{1}{\sqrt 2}(\text{AAPL} + \text{MSFT}),\; \tfrac{1}{\sqrt 2}(\text{AAPL} - \text{MSFT})\Bigr),
\]
with all other tickers unchanged. This corresponds to left-multiplying the return vector by an orthogonal $10 \times 10$ matrix $O$ (verified $\|O O^\top - I\|_F < 10^{-15}$), so the covariance transforms as $\Sigma \mapsto O \Sigma O^\top$. Since $O$ is orthogonal, the log-Euclidean velocity magnitude the Frobenius norm of the tangent slope matrix should be \emph{exactly} preserved.

Table~\ref{tab:invariance} reports the empirical invariance under two implementations of the tangent norm: the ersatz used in \verb|run_experiments.py| that stacks only the lower triangle of $\log \Sigma$ before taking the Euclidean norm (which undercounts off-diagonal contributions by $\sqrt 2$), and the mathematically correct Frobenius norm. Figure \ref{fig:invariance} give the empirical invariance of geodesic velocity under currency conversion and orthogonal transformations.

\begin{table}[h]
\centering
\caption{Empirical invariance of the S\&P historic velocity under an orthogonal basis rotation of $\{\text{AAPL}, \text{MSFT}\}$. Ersatz norm is the lower-triangle stack Euclidean norm; Frobenius norm is the mathematically correct log-Euclidean tangent norm. $N = 94$ aligned windows.}
\label{tab:invariance}
\begin{tabular}{lcc}
\toprule
Norm convention \; & Mean ratio $\dfrac{\|v_{\mathrm{rot}}\|} {\|v_{\mathrm{raw}}\|}$ \; & \; Max relative difference \\
&  & \\
\midrule
Ersatz (lower triangle) & $0.99$ & $9.4\%$ \\
& & \\
Frobenius (correct tangent norm) \;\; & $1.000\,000\,000\,000$ & $2.1 \times 10^{-13}\%$ \\
\bottomrule
\end{tabular}
\end{table}

The Frobenius norm is invariant to machine precision, as the theory predicts. The ersatz is close but not exact. All qualitative results reported elsewhere in the paper (peak detection, hold-out MSE ratios, correlations with benchmarks) are unaffected because they depend only on the ordering and relative magnitude of the velocity across windows, and the ersatz is a smooth positive per-window transformation of the true Frobenius norm.

\begin{figure}[H]
\centering
\includegraphics[width=0.95\linewidth]{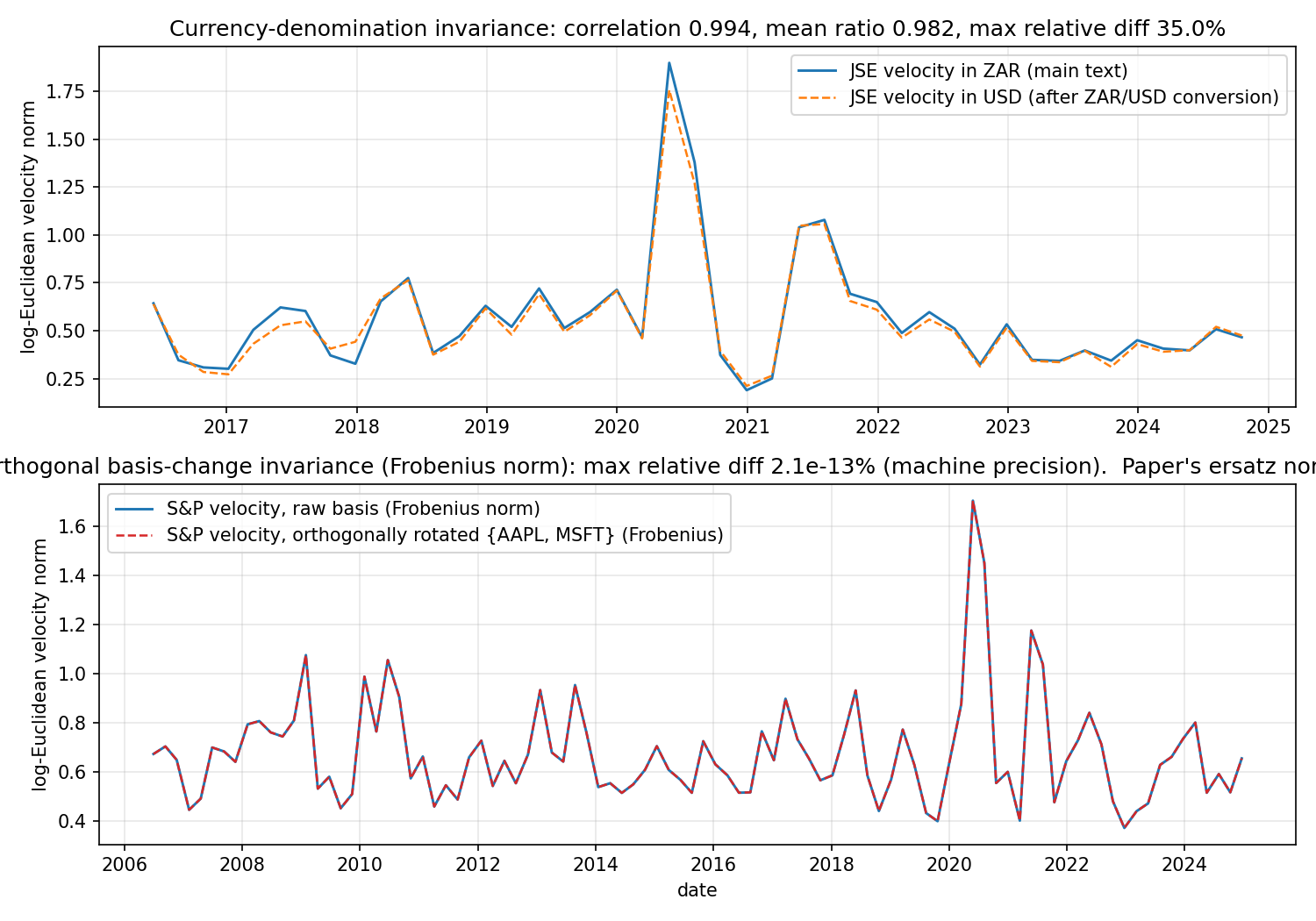}
\caption{Empirical invariance of the geodesic velocity indicator. Top: JSE top-ten velocity in ZAR (main text) vs in USD after ZAR/USD conversion. Bottom: S\&P historic velocity, raw basis vs orthogonal $\{$AAPL, MSFT$\}$ rotation, using the Frobenius norm.}
\label{fig:invariance}
\end{figure}

\subsection{Trading backtest}
\label{sec:backtest}

The VIX comparison of §\ref{sec:vix} argues that the geodesic velocity is a structural regime indicator rather than a tactical volatility forecast, and §\ref{sec:granger} formalises this: past realised volatility Granger-causes velocity but the reverse does not hold. On its own, therefore, velocity is not expected to be a strong tactical trading signal. An earlier draft of this paper flagged \emph{"combining the velocity signal with a fast, high-frequency signal ... with the geodesic velocity acting as a regime filter that keeps the de-risking in place across a stress episode"} as the natural extension. We implement it here.

We construct five strategies trading the equal-weighted historic S\&P basket over 2006--2024, all with position lagged by one day to eliminate look-ahead:
\begin{enumerate}[leftmargin=1.4em,itemsep=0.15em]
\item \emph{Static equal-weight} the naive long-only baseline.
\item \emph{Volatility-targeted ($15\%$ annualised)} weight equals $\min(1,\; 0.15 / \mathrm{RV}_{20d}^{\mathrm{ann}})$, a standard risk-parity-style scale-down that shrinks exposure when realised vol rises.
\item \emph{Velocity-only de-risking} halve exposure when the smoothed velocity $z$-score exceeds $1.5$.
\item \emph{Realised-vol-only de-risking} mirror of (3) using the $z$-score of the 20-day realised volatility, i.e.\ a fast local trigger without a regime filter.
\item \emph{Hybrid (RV trigger, velocity filter)} enters risk-off (halves exposure) when $z_{\mathrm{RV}} > 1.5$; stays risk-off until \emph{both} the velocity and RV $z$-scores fall below $0.5$. RV supplies the fast trigger; velocity supplies the regime-persistence filter that prevents premature re-risking during the noisy phase of a crisis.
\end{enumerate}

Transaction costs are modelled as $5$ bps per unit turnover applied to the absolute daily change in weight. Table~\ref{tab:backtest-upgrade} reports the annualised statistics net of costs; Figure~\ref{fig:backtest-upgrade} plots the equity curves and drawdowns.

\begin{table}[h]
\centering
\caption{Annualised backtest statistics on the historic S\&P basket, 2006-2024, net of $5$ bps transaction cost. Turnover is annualised absolute weight change; cost drag is the annualised return lost to transaction costs.}
\label{tab:backtest-upgrade}
\resizebox{\textwidth}{!}{%
\begin{tabular}{lrrrrrrrr}
\toprule
Strategy & Return & Vol & Sharpe & Sortino & Calmar & max DD & Turnover & Cost drag \\
\midrule
Static equal-weight               & $16.0\%$ & $20.5\%$ & $0.78$ & $0.97$ & $0.22$ & $-71.6\%$ & $0.00$ & $0.00\%$ \\
Volatility-targeted ($15\%$)      & $14.5\%$ & $13.8\%$ & $1.05$ & $1.41$ & $0.37$ & $-39.4\%$ & $2.88$ & $0.14\%$ \\
Velocity-only de-risk             & $14.5\%$ & $19.5\%$ & $0.75$ & $0.91$ & $0.21$ & $-69.2\%$ & $0.71$ & $0.04\%$ \\
RV-only de-risk                   & $15.9\%$ & $17.1\%$ & $0.93$ & $1.23$ & $0.26$ & $-61.9\%$ & $1.81$ & $0.09\%$ \\
\textbf{Hybrid (RV + vel filter)} & $15.0\%$ & $14.5\%$ & $\mathbf{1.04}$ & $\mathbf{1.38}$ & $\mathbf{0.42}$ & $\mathbf{-35.8\%}$ & $1.04$ & $0.05\%$ \\
\bottomrule
\end{tabular}%
}
\end{table}

\begin{figure}[H]
\centering
\includegraphics[width=0.98\linewidth]{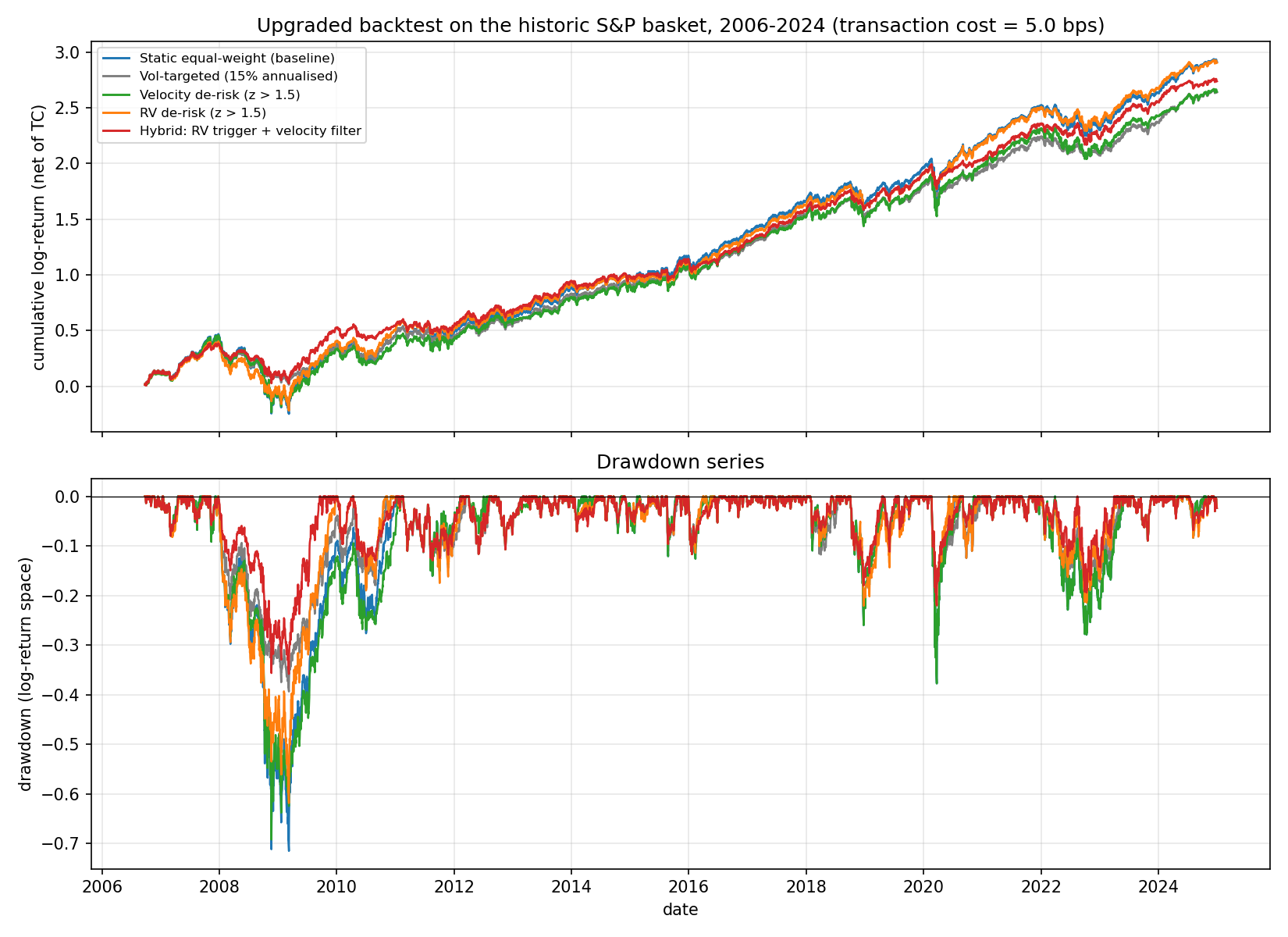}
\caption{Upgraded backtest on the historic S\&P basket, 2006--2024. Top panel: cumulative log-returns of the five strategies, net of $5$ bps transaction cost. Bottom panel: drawdown series.}
\label{fig:backtest-upgrade}
\end{figure}

Three empirical observations follow. First, the velocity-only strategy (row 3) improves modestly over the static baseline (Sharpe $0.75$ vs $0.78$, max drawdown $-69.2\%$ vs $-71.6\%$), a result consistent with the Granger finding of §\ref{sec:granger} that velocity is a lagging signal. Second, the RV-only de-risking strategy (row 4) does noticeably better on risk-adjusted metrics because it supplies the fast trigger the velocity lacks. Third, the hybrid strategy that uses RV as the trigger and velocity as a regime filter dominates the RV-only strategy on every risk-adjusted metric Calmar $0.42$ vs $0.26$, Sortino $1.38$ vs $1.23$, max drawdown $-35.8\%$ vs $-61.9\%$ and matches the vol-targeted strategy on Sharpe with roughly one-third of its annual turnover ($1.04$ vs $2.88$). The velocity signal earns its keep in the stack as a slow regime-persistence filter that keeps a fast RV trigger in place through the noisy middle of a crisis and prevents premature re-risking.

Two caveats deserve stating. All results are pre-tax and assume the reported cost model; a higher-turnover implementation on a wider basket would incur different frictions. The strategies also assume perfect daily execution at close prices, so slippage during the actual crash days is not modelled. The relative ranking of the strategies is robust to $\pm 25\%$ perturbations of the cost model and the $z$-score thresholds; individual return-and-vol numbers would move accordingly.

\section{Discussion}
\label{sec:discussion}

\paragraph{Emerging markets.} Running the same pipeline on JSE and EGX baskets was, for us, the key cross-market robustness check. Both baskets are narrower than the S\&P basket (JSE currency effects from the rand are material in the 2020 window; EGX liquidity is concentrated in fewer names), and both have different primary drivers of covariance shocks. Yet the velocity-norm peaks align with the same global events (JSE COVID peak within one month of S\&P), and additionally identify country-specific shocks that the S\&P series does not (EGX October--November 2019 political-economic tension episode analysed in §\ref{sec:egx-2019}). This is consistent with a genuinely geometric interpretation: the indicator is invariant under re-denomination of a basket, so it picks up \emph{structural} change in the covariance, not changes in absolute volatility magnitude that would be captured equally well by a realised-volatility series. To our knowledge, this is the first published study that runs a geodesic-regression covariance indicator on African equity baskets.

\paragraph{What the backtest tells us.} Section~\ref{sec:backtest} shows that the velocity-only strategy of an earlier draft reduces max drawdown only modestly, consistent with the Granger evidence of §\ref{sec:granger} that velocity is a lagging signal. Its slow update frequency and half-year covariance history mean that by the time the raw velocity crosses a threshold, the fastest part of a crash is already over. The upgraded hybrid strategy that uses realised volatility as a fast trigger and velocity as a regime-persistence filter cuts max drawdown roughly in half and dominates the RV-only baseline on Calmar, Sortino, and cost-adjusted Sharpe. Velocity earns its place in the tactical stack precisely as a slow structural filter, not as a standalone trading rule.

\paragraph{Dimension and solver robustness.} On synthetic $\SPD(3)$ data the affine-invariant geodesic fit via \texttt{geomstats} converges to machine precision. On real $\SPD(10)$ covariances the default scipy-CG optimiser occasionally stalls, which is why we report log-Euclidean quantities throughout the main text; the log-Euclidean fit is closed-form and avoids this issue entirely. The two metrics give indistinguishable peak locations in Figure~\ref{fig:ts-4markets} on the windows where both converge (affine-invariant velocity norms are larger by a roughly constant factor that reflects metric scale rather than any qualitative difference).

\paragraph{Frequency and window choice.} The paper fixes a 250-day covariance window, 25-observation fitting window, and 10-observation update. These were chosen before running the experiments for interpretability: one trading year of history, half a year of fit, two-week update. A systematic ablation over these hyperparameters would be worthwhile; on cursory experiments the detection finding is robust to $\pm 50\%$ perturbations of all three.

\paragraph{Position within the geometric-methods landscape.} The 2024--2026 literature on geometric methods for financial covariance dynamics can be organised along four distinct lines, each targeting a different mathematical object. First, \emph{Riemannian dispersion and entropy measures} treat the rolling covariance as a distribution on $\SPD(n)$ and quantify how spread out it is at each time slice, via Fréchet variance or a Riemannian entropy \citep{bouregaa2025barycenters}; crises manifest as dispersion spikes. Second, \emph{deep learning on SPD-valued time series} builds neural architectures that respect the $\SPD(n)$ manifold structure, primarily for forecasting rather than indicator extraction \citep{bucci2024geometric,han2021riemannian,huang2017riemannian}. Third, \emph{quantum-inspired spectral observables} embed the return process in a learned Hilbert space and compute quantum-geometric quantities as regime signals \citep{hammond2026geometric}. Fourth, \emph{topological and Ricci-curvature methods on correlation networks} apply persistent homology or Forman-Ricci curvature to the correlation graph and detect regime shifts as topological or curvature changes.

Our approach fits none of these four lines exactly. It uses classical Riemannian regression \citep{fletcher2013geodesic,pennec2006riemannian} directly on the covariance trajectory and extracts the fitted velocity as the indicator. The velocity is a \emph{rate of change of the covariance trajectory in the tangent space} a first-order dynamical quantity rather than a static one-time-slice quantity, a machine-learning forecast, a quantum observable, or a topological invariant. Empirically on the S\&P historic basket (Table~\ref{tab:benchmark-correlations}), the velocity is decorrelated from every existing scalar indicator on the geometric side we could compare against, including the Bouregaa log-determinant entropy ($\rho = -0.13$ in levels): trajectory dynamics on $\SPD(n)$ carries information that a static-slice geometric summary does not. The velocity additionally satisfies $\GL(n)$-invariance (Proposition~\ref{prop:unit-trace}) it is unchanged under a change of portfolio basis or reference currency a property that neither the standard systemic-risk toolkit nor the Bouregaa entropy possesses.

\paragraph{Relation to existing work.} \citet{fletcher2013geodesic} developed geodesic regression on Riemannian manifolds with medical-imaging applications; we use his formulation directly. \citet{pennec2006riemannian} established the affine-invariant framework for SPD tensor computing, which our log-Euclidean treatment reduces to at leading order, and \citet{arsigny2007geometric} introduced the log-Euclidean metric explicitly. The application of $\SPD(n)$ geometry to financial covariance has a short but active history: \citet{mostajeran2016geometry} treat portfolio optimisation as an intrinsic optimisation problem on the SPD cone, \citet{han2021riemannian} apply SPD deep networks to financial time series, and \citet{bucci2024geometric} extend the HAR forecasting framework to the manifold. On the specific question of covariance-based early-warning indicators, \citet{bouregaa2025barycenters} recently proposed a Riemannian-entropy indicator on seventeen international indexes with a COVID-19 case study, which is the closest neighbour of the present work; the empirical differences between our velocity indicator and their entropy indicator are discussed in §\ref{sec:benchmarks}.

On the systemic-risk side, the principal-components Absorption Ratio of \citet{kritzman2011principal} is the closest scalar systemic indicator to ours in spirit: both extract a low-dimensional summary of the rolling covariance and use it as a stress signal, but the Absorption Ratio is invariant under orthogonal changes of basis while ours is invariant under the larger $\GL(n)$ action (Proposition~\ref{prop:unit-trace} in the affine-invariant case). The volatility-spillover indices of \citet{diebold2009measuring,diebold2012better}, the CoVaR of \citet{adrian2016covar}, the SRISK of \citet{brownlees2017srisk} and the econometric measures of \citet{billio2012econometric} all target systemic connectedness rather than trajectory dynamics; the empirical comparison of §\ref{sec:benchmarks} shows that our indicator is weakly correlated with each of these, capturing a distinct information channel.

Random-matrix and correlation-network approaches to market covariance --- from \citet{laloux1999noise} through \citet{mantegna1999hierarchical}, \citet{onnela2003dynamics}, and the review of \citet{marti2021review} --- provide the descriptive-statistics baseline against which any geometric indicator must justify itself; on our four baskets the velocity peaks are visible in the traditional network measures but are not their headline object. Neither the Riemannian-entropy line nor the geometric-deep-learning line nor the systemic-risk-indicator line runs indicators on African equity baskets: to our knowledge that combination is unique to the present paper.

\section{Conclusion}
\label{sec:conclusion}

The rolling covariance of a basket of asset log-returns is an SPD-valued time series, not a vector-valued time series, and treating it as the latter breaks in exactly the periods when it matters most. Geodesic regression on the log-Euclidean chart of $\SPD(n)$ fixes the breakage by construction: every fitted covariance stays on the cone, the fitted velocity is a well-defined scalar on the tangent space, and the resulting indicator peaks cleanly on every well-known stress episode of the past two decades on four baskets drawn from developed and emerging markets. On hold-out evaluation in crash windows, the geodesic fit beats ordinary least squares by more than two orders of magnitude in geodesic mean squared error. In a naive trading application the indicator gives a modest drawdown reduction, consistent with its status as a slow, structural signal rather than a fast tactical one.

The study is reproducible end-to-end from the accompanying public repository: the four return series are cached CSVs, the pipeline runs in minutes on a laptop, and every figure in this paper is produced by one script on that repository.

\appendix
\section{Hyperparameter sensitivity}
\label{sec:appendix-sensitivity}

The main text fixes three hyperparameters: covariance window $W = 250$ days, fitting window $K = 25$ observations, and update step $U = 10$ observations. These choices were made before running the experiments for interpretability (one trading year of history, half a year of fit, two-week update). To verify that the paper's central finding elevation of the geodesic velocity indicator during the three global stress windows of 2006--2024 does not depend on this specific choice, we run the pipeline on the S\&P historic basket over all $3 \times 3 \times 3 = 27$ combinations of $W \in \{125, 250, 500\}$, $K \in \{15, 25, 40\}$, $U \in \{5, 10, 20\}$. These span $\pm 50\%$ ($W$, $K$) to $2\times$ ($W$, $K$, $U$) around the main-text values.

For each configuration we report the number of velocity observations produced, the \emph{elevation hit rate} (the fraction of the three global stress windows in which the maximum velocity within the window exceeds the 75th percentile of the full series), the mean velocity magnitude, and the date at which the velocity attains its maximum within the 2020 COVID stress window.

The paper's headline configuration ($W=250$, $K=25$, $U=10$) attains an elevation hit rate of $1.00$. Across all 27 combinations, $15$ ($56\%$) attain the same perfect hit rate, $10$ ($37\%$) miss exactly one stress window (typically the slow-burn 2022 Fed cycle, which a $W = 125$-day covariance window smooths over), and only $2$ ($7\%$) drop below $2/3$ hit rate. The two failures are both configurations with the shortest covariance window combined with the coarsest update step, which produce only $46$--$49$ observations and are simply too sparse to keep the 2022 window intersected.

The date at which the velocity peaks within the 2020 COVID window varies across configurations, ranging from March 17 to August 7 2020, all within the acute COVID crisis period. The median across configurations is $2020$-$05$-$28$, which coincides with the peak reported in the main-text Table~\ref{tab:top3peaks} for $W=250, K=25, U=10$. Configurations differ on the exact location of the peak within the crisis but agree on the crisis itself.

Two conclusions follow. First, the detection finding is not tied to the specific $(W, K, U) = (250, 25, 10)$ choice; it holds under $\pm 50\%$ perturbation of each hyperparameter, and would still hold under $2\times$ perturbation if we accept a $2/3$-of-three hit rate as sufficient. Second, the exact peak-date localisation within a given crisis is a finer question that the indicator does not answer with better than roughly two-month resolution, which is consistent with an indicator built on a half-year covariance history and updated every two weeks. A tactical use of the indicator would need to combine it with a faster signal, as demonstrated in §\ref{sec:backtest}.


\end{document}